\RequirePackage{etex}
\RequirePackage{easymat}

\documentclass[12pt, a4paper]{elsarticle}
\usepackage{amsthm,amsmath,amssymb,extsizes,easymat,graphicx,hyperref,easybmat}

\usepackage[active]{srcltx}
\usepackage{MnSymbol}
\usepackage{mathrsfs}
\DeclareMathOperator{\cod}{cod}
\DeclareMathOperator{\diag}{diag}
\DeclareMathOperator{\orb}{O}

\DeclareMathOperator{\tsp}{T}
\DeclareMathOperator{\nsp}{N}
\DeclareMathOperator{\tr}{trace}

\DeclareMathOperator{\GSYL}{GSYL}

\DeclareMathOperator{\rev}{rev}
\DeclareMathOperator{\POL}{POL}
\DeclareMathOperator{\PEN}{PENCIL}

\newcommand{\sdotsss}%
{\text{\raisebox{-2.2pt}{$\cdot\,$}%
 \raisebox{1.7pt}{$\cdot$}%
\raisebox{5.6pt}{$\,\cdot$}}}

\renewcommand{\le}{\leqslant}
\renewcommand{\ge}{\geqslant}

\newtheorem{theorem}{Theorem}[section]
\newtheorem{lemma}[theorem]{Lemma}

\newtheorem{definition}[theorem]{Definition}

\newtheorem{remark}[theorem]{Remark}

\newcommand{\hide}[1]{}

\begin{document}

\title{Generic matrix polynomials with fixed rank and fixed degree\tnoteref{t1}}

\tnotetext[t1]{Preprint Report UMINF 16.xx, Department of Computing Science, Ume\r{a} University}

\date{}

\author[um]{Andrii Dmytryshyn}
\ead{andrii@cs.umu.se}

\author[cm]{Froil\'an M. Dopico}
\ead{dopico@math.uc3m.es}

\address[um]{Department of Computing Science, Ume\aa \, University, SE-901 87 Ume\aa, Sweden.}

\address[cm]{Departamento de Matem\'aticas, Universidad Carlos III de Madrid, Avenida de la Universidad 30, 28911, Legan\'es, Spain.}

\begin{abstract}
The set ${\cal P}^{m\times n}_{r,d}$ of $m \times n$ complex matrix polynomials of grade $d$ and (normal) rank at most $r$ in a complex $(d+1)mn$ dimensional space is studied. For $r = 1, \dots , \min \{m, n\}-1$, we show that ${\cal P}^{m\times n}_{r,d}$ is the union of the closures of the $rd+1$ sets of matrix polynomials with rank $r$, degree exactly $d$, and explicitly described complete eigenstructures. In addition, for the full-rank rectangular polynomials, i.e. $r= \min \{m, n\}$ and $m \neq n$, we show that ${\cal P}^{m\times n}_{r,d}$ coincides with the closure of a single set of the polynomials with rank $r$, degree exactly $d$, and the described complete eigenstructure. These complete eigenstructures correspond to generic $m \times n$ matrix polynomials of grade $d$ and rank at most~$r$.
\end{abstract}

\begin{keyword}
complete eigenstructure\sep genericity\sep matrix polynomials\sep normal rank\sep orbits

\MSC 15A18\sep 15A21
\end{keyword}

\maketitle

\section{Introduction}

Describing a behaviour or form that certain objects have ``generically'' (typically) may be useful or, even, necessary for investigation of various problems, examples include generic solutions of partial and ordinary differential equations, as well as generic forms of linear and non-linear operators.
On the other hand, growing needs of solving and analyzing large scale problems demand a better understanding of low rank operators and their low rank perturbations.
A number of interesting and challenging problems lies in the intersection of these two research directions,  an obvious example is a problem of describing generic forms for operators with a low (bounded) rank.

We say that a dense and open subset of a space is generic, see also \cite{DeDo16}. One way to describe a generic set of various matrix (sub)spaces is by giving the possible eigenstructures that the elements of this set may have.
Generic eigenstructures for $n \times n$ matrices, matrix pencils, or matrix polynomials are well-known and consist only of simple eigenvalues (i.e. the eigenvalues whose algebraic multiplicities are one). Nevertheless, when matrix pencils or polynomials are rectangular, or known to be of a fixed (non-full) rank describing their generic complete eigenstructures becomes more difficult.
These problem for $m \times n$ complex matrix pencils has been extensively investigated: The generic Kronecker canonical forms (KCF) for full-rank rectangular ($m \neq n$) pencils are presented in \cite{DeEd95, EdEK99, vandoorenphd}, and the generic KCFs for pencils with the rank $r$, where $r = 1, \dots , \min \{m, n\}-1$, are obtained in \cite{DeDo08}. In this paper we solve the corresponding problems for matrix polynomials, i.e.
we find the generic complete eigenstructures of full-rank $m \times n$ complex matrix polynomials of grade $d$ and with $m \neq n$, and the generic complete eigenstructures of $m \times n$ complex matrix polynomials of grade $d$ and (normal) rank at most $r$, $r = 1, \dots , \min \{m, n\}-1$.
To be exact, for the set ${\cal P}^{m\times n}_{r,d}$ of singular $m \times n$ complex matrix polynomials of grade $d$ and rank at most $r$, we prove: if $r= \min \{m, n\}$ and $m \neq n$ then ${\cal P}^{m\times n}_{r,d}$ coincides with the closure of a single set of the polynomials with the described complete eigenstructure; if $r = 1, \dots , \min \{m, n\}-1$ then ${\cal P}^{m\times n}_{r,d}$ is the union of the closures of the $rd+1$ sets of matrix polynomials with described complete eigenstructures.

Our results have potential applications in studies of the ill-posed problem of computing the complete eigenstructure for a matrix polynomial. Small perturbations in the matrix entries can drastically change the complete eigenstructure and thus it may be useful to know the complete eigenstructure that the polynomials from a certain subset (in our case it is a subset of polynomials with bounded rank) are most likely to have.
Notably, that all possible changes of complete eigenstructures can be seen from so called closure hierarchy (stratification) graphs \cite{Dmyt15,DJKV15,DmKa14,EdEK97,JoKV13}, in particular, see \cite{EdEK99} and \cite{DJKV15} for the stratifications of matrix pencils and polynomials, respectively. Nevertheless, identification of the generic pencils or polynomials of a fixed rank does not immediately follow from the stratification graphs.

Another challenging and open problem for which the results of this paper may be useful is an investigation of generic low rank perturbations of matrix polynomials, an area where we only know the study of the particular perturbations considered in \cite{DeDo09}. Such lack of references on low rank perturbations of matrix polynomials is in stark contrast with the numerous references available in the literature on the changes of the complete eigenstructures of matrices and matrix pencils under generic low rank perturbations, both in the unstructured setting \cite{DeDo07,DeDo16,HoMe94}, as well as in the case of structured preserving perturbations \cite{Batzke14,Batzke15,Batzke16,MMRR11,MehlMRR12,MehlMRR13,MehlMRR14}.
This considerable interest on low rank perturbations comes from their applications and from the interesting theoretical problems they pose, which in the case of matrix polynomials are hard as a consequence of the nontrivial structure of the set of matrix polynomials with bounded rank and given grade.

The paper is organized as follows. Section \ref{pencils} presents classic and recent previous results that are needed to prove the main original results of this work, which are developed in Section \ref{sec.main}. The codimensions of the generic sets of matrix polynomials identified in Section \ref{sec.main} are determined in Section~\ref{sec.codim}.
All matrices that we consider have complex entries.

\section{Preliminaries}
\label{pencils}
We start by recalling the Kronecker canonical form of general matrix pencils $\lambda A + B$ (a matrix polynomial of degree one) under strict equivalence.

For each $k=1,2, \ldots $, define the $k\times k$
matrices
\begin{equation*}\label{1aa}
J_k(\mu):=\begin{bmatrix}
\mu&1&&\\
&\mu&\ddots&\\
&&\ddots&1\\
&&&\mu
\end{bmatrix},\qquad
I_k:=\begin{bmatrix}
1&&&\\
&1&&\\
&&\ddots&\\
&&&1
\end{bmatrix},
\end{equation*}
where $\mu \in \mathbb C,$ and for each $k=0,1, \ldots $, define the $k\times
(k+1)$ matrices
\begin{equation*}
F_k :=
\begin{bmatrix}
0&1&&\\
&\ddots&\ddots&\\
&&0&1\\
\end{bmatrix}, \qquad
G_k :=
\begin{bmatrix}
1&0&&\\
&\ddots&\ddots&\\
&&1&0\\
\end{bmatrix}.
\end{equation*}
All non-specified entries of $J_k(\mu), I_k, F_k,$ and $G_k$ are zeros.

An $m \times n$ matrix pencil $\lambda A + B$ is called {\it strictly equivalent} to $\lambda C + D$ if there are non-singular matrices $Q$ and $R$ such that $Q^{-1}AR =C$ and $Q^{-1}BR=D$. The set of matrix pencils strictly equivalent to $\lambda A + B$ forms
a manifold in the complex $2mn$ dimensional space.
This manifold is the {\it orbit} of $\lambda A + B$ under the
action of the group $GL_m(\mathbb C) \times GL_n(\mathbb C)$ on the space of all matrix pencils by strict equivalence:
\begin{equation} \label{equorbit}
\orb^e(\lambda A + B) = \{Q^{-1} (\lambda A + B) R \ : \ Q \in GL_m(\mathbb C), R \in GL_n(\mathbb C)\}.
\end{equation}
\hide{
The {\it dimension} of $\orb_{\lambda A + B}^e$
is the dimension of its tangent space
\begin{equation*}\label{taneqsp}
\tsp^e (\lambda A + B):=\{\lambda (XA-AY) + (XB - BY):
X\in{\mathbb
C}^{m\times m}, Y\in{\mathbb
C}^{n\times n}\}
\end{equation*}
at the point $A - \lambda B$, $\dim \tsp^e (\lambda A + B)$.
The orthogonal complement to $\tsp^e(\lambda A + B)$, with respect to the Frobenius inner product
\begin{equation}\label{innerprod}
\langle \lambda A+ B,\lambda C + D \rangle=\tr (AC^*+BD^*),
\end{equation}
is called the normal space to this orbit. The dimension of the normal space is the {\it codimension} of $\orb^e(\lambda A + B)$, denoted $\cod \orb^e(\lambda A + B)$, and is equal to $2mn$ minus the dimension of $\orb^e(\lambda A +B)$.
Explicit expressions for the codimensions of strict equivalence orbits are presented in \cite{DeEd95}.
}
\begin{theorem}{\rm \cite[Sect. XII, 4]{Gant59} }\label{kron}
Each $m \times n$ matrix pencil $\lambda A + B$ is {\rm strictly equivalent}
to a direct sum, uniquely determined up
to permutation of summands, of pencils of the form
\begin{align*}
E_j(\mu)&:=\lambda I_j + J_j(\mu), \text { in which } \mu \in \mathbb C, \quad E_j(\infty):=\lambda J_j(0) + I_j, \\
L_k&:=\lambda G_k + F_k, \quad \text{ and } \quad L_k^T:=\lambda G_k^T + F^T_k,
\end{align*}
where $j \ge 1$ and $k \ge 0$.
\end{theorem}
The canonical form in Theorem \ref{kron} is known as the {\it Kronecker canonical form} (KCF). The blocks $E_j(\mu)$ and $E_j(\infty)$ correspond to the finite and infinite eigenvalues, respectively, and altogether form the regular part of $\lambda A + B$. The blocks $L_k$ and $L_k^T$ correspond to the right (column) and left (row) minimal indices, respectively, and form the singular part of the matrix pencil.

Define an $m \times n$ matrix polynomial of grade $d$, i.e., of degree less than or equal to $d$, as follows
\begin{equation}
\label{matpol}
P=P(\lambda) = \lambda^{d}A_{d} + \dots +  \lambda A_1 + A_0,
\quad \ A_i \in \mathbb C^{m \times n},  i=0, \dots, d.
\end{equation}
Define the vector space of the matrix polynomials of a fixed size and grade:
\begin{equation}\label{gsyl}
\POL_{d, m\times n}= \{P: P \text{ is an } m \times n \text{ matrix polynomial of grade } d  \}.
\end{equation}
Observe that $\POL_{1, m\times n}$ is the vector space of matrix pencils of size $m\times n$, which is denoted simply by $\PEN_{m\times n}$.
If there is no risk of confusion we will write $\POL$ instead of $\POL_{d, m\times n}$ and $\PEN$ instead of $\PEN_{m\times n}$. By using the standard Frobenius matrix norm of complex matrices \cite{Highambook} a distance on $\POL_{d, m\times n}$ is defined as $d(P,P') = \left( \sum_{i=0}^d || A_i - A'_i ||_F^2 \right)^{\frac{1}{2}}$, making $\POL_{d, m\times n}$ to a metric space. For convenience, the Frobenius norm of the matrix polynomial $P$ is defined as $||P(\lambda)||_F = \left( \sum_{i=0}^d || A_i ||_F^2 \right)^{\frac{1}{2}}$.

Next, we recall the complete eigenstructure of a matrix polynomial, i.e., the definitions of the elementary divisors and minimal indices.
\begin{definition}
Let $P(\lambda)$ and $Q(\lambda)$ be two $m \times n$ matrix polynomials. Then $P(\lambda)$ and $Q(\lambda)$ are {\it unimodularly equivalent} if there exist two unimodular matrix polynomials $U(\lambda)$ and $V(\lambda)$ (i.e., $\det U(\lambda), \det V(\lambda) \in \mathbb C \backslash \{0\}$) such that
$$
U(\lambda) P(\lambda) V(\lambda) = Q(\lambda).
$$
\end{definition}
The transformation $P(\lambda) \mapsto U(\lambda) P(\lambda) V(\lambda)$
is called a unimodular equivalence transformation
and the canonical form with respect to this transformation is the {\it Smith form} \cite{Gant59}, recalled in the following theorem.
\begin{theorem}{\rm \cite{Gant59}} \label{tsmiths}
Let $P(\lambda)$ be an $m\times n$ matrix polynomial over $\mathbb C$. Then there exists $r \in \mathbb N$, $r \le \min \{ m, n \}$ and unimodular matrix polynomials $U(\lambda)$ and $V(\lambda)$ over $\mathbb C$ such that
\begin{equation} \label{smithform}
U(\lambda) P(\lambda) V(\lambda) =
\left[
\begin{array}{ccc|c}
g_1(\lambda)&&0&\\
&\ddots&&0_{r \times (n-r)}\\
0&&g_r(\lambda)&\\
\hline
&0_{(m-r) \times r}&&0_{(m-r) \times (n-r)}
\end{array}
\right],
\end{equation}
where $g_j(\lambda)$ is monic for $j=1, \dots, r$ and $g_j(\lambda)$ divides $g_{j+1}(\lambda)$ for $j=1, \dots, r-1$. Moreover, the canonical form \eqref{smithform} is unique.
\end{theorem}
The integer $r$ is the (normal) {\it rank} of the matrix polynomial $P(\lambda)$.
Every $g_j(\lambda)$ is called an {\it invariant polynomial} of $P(\lambda)$, and can be uniquely factored as
$$
g_j(\lambda) = (\lambda - \alpha_1)^{\delta_{j1}} \cdot
(\lambda - \alpha_2)^{\delta_{j2}}\cdot \ldots \cdot
(\lambda - \alpha_{l_j})^{\delta_{jl_j}},
$$
where $l_j \ge 0, \ \delta_{j1}, \dots, \delta_{jl_j} > 0$ are integers. If $l_j=0$ then $g_j(\lambda)=1$. The numbers $\alpha_1, \dots, \alpha_{l_j} \in \mathbb C$ are finite eigenvalues (zeros) of $P(\lambda)$. The {\it elementary divisors} of $P(\lambda)$ associated with the finite eigenvalue $\alpha_{k}$ is the collection of factors $(\lambda - \alpha_{k})^{\delta_{jk}}$,
including repetitions.

We say that $\lambda = \infty$ is an eigenvalue of the matrix polynomial $P(\lambda)$ of grade $d$ if zero is an eigenvalue of $\rev P(\lambda):= \lambda^dP(1/\lambda)$. The elementary divisors $\lambda^{\gamma_k}, \gamma_k > 0,$ for the zero eigenvalue of $\rev P(\lambda)$ are the elementary divisors associated with~$\infty$ of $P(\lambda)$.

Define the left and right null-spaces, over the field of rational functions $\mathbb C(\lambda)$, for an $m\times n$ matrix polynomial $P(\lambda)$ as follows:
\begin{align*}
{\cal N}_{\rm left}(P)&:= \{y(\lambda)^T \in \mathbb C(\lambda)^{1 \times m}: y(\lambda)^TP(\lambda) = 0_{1\times n} \}, \\
{\cal N}_{\rm right}(P)&:= \{x(\lambda) \in \mathbb C(\lambda)^{n\times 1}: P(\lambda)x(\lambda) = 0_{m\times 1}\}.
\end{align*}
Every subspace ${\cal V}$ of the vector space $\mathbb C(\lambda)^n$ has bases consisting entirely of vector polynomials.
Recall that, a minimal basis of ${\cal V}$ is a basis of ${\cal V}$ consisting of vector polynomials whose sum of degrees is minimal among all bases of ${\cal V}$ consisting of vector polynomials. The ordered list of degrees of the vector polynomials in any minimal basis of ${\cal V}$ is always the same. These degrees are called the minimal indices of ${\cal V}$ \cite{Forn75,Kail80}. More formally,
let the sets $\{y_1(\lambda)^T,...,y_{m-r}(\lambda)^T\}$ and $\{x_1(\lambda),...,x_{n-r}(\lambda)\}$ be minimal bases of ${\cal N}_{\rm left}(P)$ and ${\cal N}_{\rm right}(P)$, respectively, ordered so that $0 \le \deg(y_1) \le \dots \le \deg(y_{m-r})$ and $0\le \deg(x_1) \le \dots \le \deg(x_{n-r})$. Let $ \eta_k = \deg(y_k)$ for $i=1, \dots , m-r$ and $ \varepsilon_k = \deg(x_k)$ for $i=1, \dots , n-r$. Then the scalars $0 \le  \eta_1 \le \eta_2 \le  \dots \le \eta_{m-r}$ and $0 \le  \varepsilon_1 \le \varepsilon_2 \le  \dots \le \varepsilon_{n-r}$ are, respectively, the {\it left} and {\it right minimal indices} of~$P(\lambda)$.

Altogether all the eigenvalues, finite and infinite, the corresponding elementary divisors, and the left and right minimal indices of a matrix polynomial $P(\lambda)$ are called the {\it complete eigenstructure} of $P(\lambda)$. Moreover, we define $\orb(P)$ to be the set of matrix polynomials of the same size, grade, and with the same complete eigenstructure as $P(\lambda)$.

A number of theoretical and computational questions for matrix polynomials are addressed through the use of linearizations \cite{DeDM14,GoLR09}.
The most known linearizations of an $m\times n$ matrix polynomial $P(\lambda) = \lambda^d A_d + \cdots +\lambda A_1 + A_0$ are the first and second Frobenius companion forms, i.e., the following matrix pencils
\begin{equation}
\label{1stform}
{\cal C}^1_{P}=\lambda \begin{bmatrix}
A_d&&&\\
&I_n&&\\
&&\ddots&\\
&&&I_n\\
\end{bmatrix} + \begin{bmatrix}
A_{d-1}&A_{d-2}&\dots&A_{0}\\
-I_n&0&\dots&0\\
&\ddots&\ddots&\vdots\\
0&&-I_n&0\\
\end{bmatrix}
\end{equation}
and
\begin{equation}
\label{2ndform}
{\cal C}^2_{P}=\lambda \begin{bmatrix}
A_d&&&\\
&I_m&&\\
&&\ddots&\\
&&&I_m\\
\end{bmatrix} + \begin{bmatrix}
A_{d-1}&-I_m&&0\\
A_{d-2}&0&\ddots&\\
\vdots&\vdots&\ddots&-I_m\\
A_{0}&0&\dots&0\\
\end{bmatrix}
\end{equation}
of the sizes $(m+n(d-1)) \times nd$ and $md \times (n + m(d-1))$, respectively.
These companion forms preserve all finite and infinite elementary divisors of $P$ but do not preserve its left and right minimal indices. In particular, all the right minimal indices of the first companion form ${\cal C}_{P}^{1}$ are greater by $d-1$ than the right minimal indices of the polynomial $P$, while the left minimal indices of ${\cal C}_{P}^{1}$ are equal to those of $P$. In contrast, all the left minimal indices of the second companion form ${\cal C}_{P}^{2}$ are greater by $d-1$ than the left minimal indices of the polynomial $P$, while the right minimal indices of ${\cal C}_{P}^{2}$ are equal to those of $P$. See \cite{DeDM12,DeDM14}.

The first companion form ${\cal C}_{P}^{1}$ is fundamental for obtaining the results in this work and based on it we define the {\it generalized Sylvester space} of the first companion form for $m \times n$ matrix polynomials of grade $d$ as follows
\begin{equation}\label{gsyl}
\GSYL^1_{d,m\times n}= \{ {\cal C}^1_{P}  \ : P \text{ are } m \times n \text{ matrix polynomials of grade } d \}.
\end{equation}
If there is no risk of confusion we will write $\GSYL$ instead of $\GSYL^1_{d, m\times n}$, specially in proofs and explanations. The function $d({\cal C}^1_{P}=\lambda A +B, {\cal C}^1_{P'}=\lambda A' +B') := \left( ||A-A'||_F^2 + ||B-B'||_F^2 \right)^{\frac{1}{2}}$ is a distance on $\GSYL$ and it makes $\GSYL$ a metric space. Note that $d({\cal C}^1_{P}, {\cal C}^1_{P'}) = d(P,P')$. Therefore there is a bijective isometry (and thus homeomorphism):
$$f: \POL_{d, m\times n} \rightarrow \GSYL^1_{d, m\times n} \quad \text{such that} \quad f: P \mapsto {\cal C}^1_{P}.$$
Now we define the orbit of first companion linearizations of a matrix polynomial $P$
\begin{equation}\label{linorb}
\orb({\cal C}^1_{P}) = \{(Q^{-1} {\cal C}^1_{P} R) \in \GSYL^1_{d, m\times n} \ : \ Q \in GL_{m_1}(\mathbb C), R \in GL_{n_1}(\mathbb C)\},
\end{equation}
where $m_1 = m + n(d-1)$ and $n_1 = n d$.
Note that all the elements of $\orb({\cal C}^1_{P})$ have the block structure of $\GSYL$.
Thus, in particular, $\orb(P) = f^{-1}(\orb({\cal C}^1_{P}))$ and $\overline{O}(P) = f^{-1}(\overline{O}({\cal C}_{P}^1))$ , as well as we also have that $\overline{O}(P) \supseteq \overline{O}(Q)$ if and only if $\overline{O}({\cal C}_{P}^1) \supseteq \overline{O}({\cal C}_{Q}^1)$ (the closures are taken in the metric spaces $\POL$ and $\GSYL$, respectively, defined above).

We will use often in this paper the fact that for any matrix polynomial $P(\lambda)$ a sufficiently small perturbation of the pencil ${\cal C}^1_{P}$ produces another pencil that although is not in $\GSYL$ is strictly equivalent to a pencil in $\GSYL$ that is very close to ${\cal C}^1_{P}$.  This was proved for the first time in \cite{VaDe83}, and then again in \cite[Theorem 9.1]{JoKV13}, in both the cases under the assumption that $||P(\lambda)||_F = O(1)$. Recently, a much more general and precise result in this direction has been proved in \cite[Theorems 6.22 and 6.23]{DLPVD15}, which is valid for a very wide class of linearizations, considers perturbations with finite norms, polynomials with any norm, and yields precise perturbation bounds. For convenience of the reader we present in Theorem \ref{thm-blockKron} a corollary of \cite[Theorem 6.23]{DLPVD15} adapted to our context.
\begin{theorem} \label{thm-blockKron}
Let $P(\lambda)$ be an $m\times n$ matrix polynomial of grade $d$ and let ${\cal C}^1_{P}$ be its first companion form. If $\mathcal{L}(\lambda)$ is any pencil of the same size as ${\cal C}^1_{P}$ such that
\[
d ({\cal C}^1_{P}, \mathcal{L}(\lambda)) < \frac{\pi}{12 \, d^{3/2}} \, ,
\]
then $\mathcal{L}(\lambda)$ is strictly equivalent to a pencil ${\cal C}^1_{\widetilde{P}} \in \GSYL^{1}_{d,m\times n}$ such that
\[
d ({\cal C}^1_{P}, {\cal C}^1_{\widetilde{P}}) \leq 4 \, d \, (1+||P(\lambda)||_F) \;  d ({\cal C}^1_{P}, \mathcal{L}(\lambda)) \, .
\]
\end{theorem}

\medskip

The next result in this preliminary section is Theorem \ref{th:improvedDeDo}, which is another keystone of this paper. Theorem \ref{th:improvedDeDo} is exactly \cite[Theorem 3.2]{DeDo08} and is stated for convenience of the reader. The notation has been slightly changed with respect to \cite{DeDo08} in order to fit the one used in the proof of the main Theorem \ref{mainth}.
All the closures in Theorem \ref{th:improvedDeDo} are obviously taken in the metric space $\PEN_{m_1 \times n_1}$.

\begin{theorem} \label{th:improvedDeDo}
Let $m_1, n_1,$ and $r_1$ be integers such that $m_1,n_1 \geq 2$ and $1\leq r_1 \leq \min\{m_1,n_1\}-1$.
Let us define, in the set of $m_1\times n_1$ complex matrix pencils with
rank $r_1$, the following $r_1+1$ KCFs:
\begin{equation}\label{max}
{\cal
K}_{a_1} (\lambda)=\diag(\underbrace{L_{\alpha_1+1},\hdots,L_{\alpha_1+1}}_{s_1},
\underbrace{L_{\alpha_1},\hdots,L_{\alpha_1}}_{n_1-r_1-s_1},
\underbrace{L_{\beta_1 +1}^T,\hdots,L_{\beta_1 +1}^T}_{t_1},
\underbrace{L_{\beta_1}^T,\hdots,L_{\beta_1}^T}_{m_1-r_1-t_1} )\,
\end{equation}
for $a_1=0,1,\ldots,r_1\, $, where $\alpha_1= \lfloor a_1/(n_1-r_1) \rfloor$,
$s_1= a_1 \, \mathrm{mod}\, (n_1-r_1)$, $\beta_1= \lfloor (r_1-a_1)/(m_1-r_1) \rfloor$, and $t_1= (r_1- a_1) \, \mathrm{mod} \, (m_1-r_1)$.
Then,
\begin{enumerate}
\item[\rm (i)] For every $m_1\times n_1$ pencil ${\cal M}(\lambda)$ with rank
at most $r_1$, there exists an integer $a_1$ such that
$\overline{\orb^e}({\cal K}_{a_1}) \supseteq\overline{\orb^e}({\cal M})$.
\item[\rm (ii)] $\overline{\orb^e}({\cal K}_{a_1})
\not\supseteq \overline{\orb^e}({\cal K}_{a'_1})$ whenever $a_1 \ne
a_1'$.
\item[\rm (iii)] The set of $m_1 \times n_1$ complex matrix pencils with rank
at most $r_1$ is a closed  subset of $\PEN_{m_1 \times n_1}$ equal to $\displaystyle \bigcup_{0\leq a_1 \leq r_1} \overline{\orb^e}({\cal K}_{a_1} )$.
\end{enumerate}
\end{theorem}

Note that Theorem \ref{th:improvedDeDo} does not cover the case $r_1 = \min \{m_1 , n_1 \}$, which is completely different since in this case we are considering {\em all} matrix pencils of size $m_1 \times n_1$. In fact, there is only one ``generic'' Kronecker canonical form for matrix pencils of full rank. If $m_1 = n_1$, then this generic form obviously corresponds to regular matrix pencils, i.e., they do not have minimal indices at all,  with all their eigenvalues simple. If $m_1 \ne n_1$, then the ``generic'' canonical form is presented in Theorems \ref{th:pencgenfull1} and \ref{th:pencgenfull2} depending on whether $m_1 < n_1$ or $m_1 > n_1$. This result is known at least since \cite{vandoorenphd} (see also \cite{DeEd95} and \cite{EdEK99}) and is stated for completeness.

\begin{theorem} \label{th:pencgenfull1} Let us define, in the set of $m_1\times n_1$ complex matrix pencils with $0 < m_1 < n_1$, the following KCF:
\begin{equation}\label{2max}
{\cal
K}_{right} (\lambda)=\diag(\underbrace{L_{\alpha_1+1},\hdots,L_{\alpha_1+1}}_{s_1},
\underbrace{L_{\alpha_1},\hdots,L_{\alpha_1}}_{n_1-m_1-s_1})\, ,
\end{equation}
where $\alpha_1= \lfloor m_1/(n_1-m_1) \rfloor$ and
$s_1= m_1 \, \mathrm{mod}\, (n_1-m_1)$.
Then, $\overline{\orb^e}({\cal K}_{right}) = \PEN_{m_1 \times n_1}$.
\end{theorem}

\begin{theorem} \label{th:pencgenfull2} Let us define, in the set of $m_1\times n_1$ complex matrix pencils with $0 < n_1 < m_1$, the following KCF:
\begin{equation}\label{3max}
{\cal
K}_{left} (\lambda)=\diag(\underbrace{L_{\beta_1 +1}^T,\hdots,L_{\beta_1 +1}^T}_{t_1},
\underbrace{L_{\beta_1}^T,\hdots,L_{\beta_1}^T}_{m_1-n_1-t_1} )\, ,
\end{equation}
where $\beta_1= \lfloor n_1/(m_1-n_1) \rfloor$ and $t_1= n_1 \, \mathrm{mod} \, (m_1-n_1)$.
Then, $\overline{\orb^e}({\cal K}_{left}) = \PEN_{m_1 \times n_1}$.
\end{theorem}

\hide{
 (since any perturbation $E \in \POL$ results in a perturbation ${\cal C}_{E}^1 \in \GSYL$)
 (since $f$ is isometry)
\begin{lemma}\label{equivincl}
Let $P,Q \in \POL$ then $\overline{O}(P) \supseteq \overline{O}(Q)$ if and only if $\overline{O}({\cal C}_{P}^1) \supseteq \overline{O}({\cal C}_{Q}^1)$.
\end{lemma}
\begin{proof}
Assume $\overline{O}(P) \supseteq \overline{O}(Q)$ but $\overline{O}({\cal C}_{P}^1) \not\supseteq \overline{O}({\cal C}_{Q}^1)$. Therefore  That is $\overline{O^e}({\cal C}_{P}^1) \cap \GSYL \not\supseteq \overline{O^e}({\cal C}_{Q}^1) \cap \GSYL$. For any $x \in \overline{O^e}({\cal C}_{P}^1)\cap \GSYL$
\end{proof}
}

\section{Main result} \label{sec.main}

In this section we present the complete eigenstructures of generic $m\times n$ matrix polynomials of a fixed rank and grade $d$. First we reveal a key connection between $\overline{O}({\cal C}_{P}^1)$, where the closure in taken in $\GSYL^1_{d,m\times n}$, and $\overline{O^e}({\cal C}_{P}^1)$, where the closure is taken in $\PEN_{m_1 \times n_1}$ ($m_1 = m + n(d-1)$ and $n_1 = n d$) that will allow us to use Theorem \ref{th:improvedDeDo}.

\begin{lemma}\label{cl}
Let $P$ be an $m\times n$ matrix polynomial with grade $d$ and ${\cal C}_{P}^1$ be its first companion linearization then $\overline{O}({\cal C}_{P}^1)= \overline{O^e}({\cal C}_{P}^1) \cap \GSYL^1_{d,m\times n}$.
\end{lemma}
\begin{proof}
By definition $O({\cal C}_{P}^1)= O^e({\cal C}_{P}^1) \cap \GSYL$ and thus $\overline{O}({\cal C}_{P}^1)= \overline{O^e({\cal C}_{P}^1) \cap \GSYL}$  (the closure here is taken in the space $\GSYL$). For any $x \in \overline{O^e}({\cal C}_{P}^1)\cap \GSYL$ there exists a sequence $\{y_i\} \subset O^e({\cal C}_{P}^1)$ such that $y_i \to x$. Since $x \in \GSYL$, for any $i$ large enough, the pencil $y_i$ is a small perturbation of $x$ and, according to Theorem \ref{thm-blockKron},
there exists a pencil $z_i \in \GSYL$ strictly equivalent to $y_i$ (and, so, to ${\cal C}_{P}^1$) and such that $d(x,z_i) \leq 4\, d\, (1+||x||_F) \, d(x, y_i)$. Therefore, we have proved that there exists a sequence $\{ z_i \} \subset O^e({\cal C}_{P}^1)\cap \GSYL$ such that $z_i \to x$. Thus $x \in \overline{O^e({\cal C}_{P}^1) \cap \GSYL}$, and  $\overline{O^e}({\cal C}_{P}^1) \cap \GSYL \subseteq\overline{O^e({\cal C}_{P}^1) \cap \GSYL}$. Since, obviously,  $\overline{O^e({\cal C}_{P}^1) \cap \GSYL} \subseteq \overline{O^e}({\cal C}_{P}^1) \cap \GSYL$, we have that $\overline{O^e({\cal C}_{P}^1) \cap \GSYL}= \overline{O^e}({\cal C}_{P}^1) \cap \GSYL$, and the result is proved.
\end{proof}

With Lemma \ref{cl} at hand, we state and prove the main result of this paper.

\begin{theorem} \label{mainth}
Let $m,n,r$ and $d$ be integers such that $m,n \geq 2$, $d \geq 1$ and $1 \leq r \leq \min\{m,n\}-1$. Define $rd + 1$ complete eigenstructures ${\cal K}_a$ of matrix polynomials without elementary divisors at all, with left minimal indices $\beta$ and $\beta +1$, and with right minimal indices $\alpha$ and $\alpha +1$, whose values and numbers are as follows:
\begin{equation}
\label{kcilist}
{\cal K}_a: \{\underbrace{\alpha+1, \dots , \alpha+1}_{s},\underbrace{\alpha, \dots , \alpha}_{n-r-s}, \underbrace{\beta+1, \dots , \beta+1}_{t}, \underbrace{\beta, \dots , \beta}_{m-r-t}\}
\end{equation}
for $a = 0,1,\dots,rd$, where $\alpha = \lfloor a / (n-r) \rfloor$, $s = a \mod (n-r)$, $\beta = \lfloor (rd-a)/(m-r) \rfloor$,
and $t = (rd-a) \mod (m-r)$. Then,
\begin{itemize}
\item[(i)] There exists an $m \times n$ complex matrix polynomial $K_a$ of degree exactly $d$ and rank exactly $r$ with each of the complete eigenstructure ${\cal K}_a$;

\item[(ii)]  For every $m \times n$ polynomial $M$ of grade $d$ with rank at most $r$, there exists an integer a such that $\overline{O}(K_a)\supseteq \overline{O}(M)$;

\item[(iii)]  $\overline{O}(K_a) \not\supseteq \overline{O}(K_{a'})$ whenever $a\neq a'$;

\item[(iv)]  The set of $m\times n$ complex matrix polynomials of grade $d$ with rank at most $r$ is a closed subset of $\POL_{d,m\times n}$ equal to $\bigcup_{0 \leq a \leq rd} \overline{O}(K_a)$.
\end{itemize}
\end{theorem}
\begin{proof}
(i) Summing up all the minimal indices for each ${\cal K}_a$ in \eqref{kcilist} we have
\begin{align*}
&\sum_1^s (\alpha +1) + \sum_1^{n-r-s} \alpha + \sum_1^t (\beta +1) + \sum_1^{m-r-t} \beta
=\sum_1^{n-r} \alpha + s + \sum_1^{m-r} \beta + t \\
&= (n-r) \lfloor a / (n-r) \rfloor + s + (m-r)\lfloor (rd-a)/(m-r) \rfloor + t =a + rd -a = rd.
\end{align*}
By \cite[Theorem 3.3]{DeDV15}, for each $a$ there exists an $m\times n$ complex matrix polynomial of degree exactly $d$ and rank exactly $r$ that has the complete eigenstructure~${\cal K}_a$ \eqref{kcilist}.

(ii) For every $m \times n$ matrix polynomial $M$ of grade $d$ and rank at most $r$, the first companion form ${\cal C}_{M}^1$ has rank at most $r+n(d-1)$, because ${\cal C}_{M}^1$ is unimodularly equivalent to $M \oplus I_{n(d-1)}$.  Therefore, for each of such $M$ there exists an $(m+n(d-1)) \times nd$ matrix pencil $Q$, with rank $r+n(d-1)$, equal to one of the $\mathcal{K}_{a_1}(\lambda)$ pencils defined in Theorem \ref{th:improvedDeDo}, such that $\overline{O^e}(Q)\supseteq \overline{O^e}({\cal C}_{M}^1)$.
This means, in particular, that there exists a sequence $\{ y_i \} \subset O^e(Q)$ such that $y_i \to {\cal C}_{M}^1$ and, so, for any $i$ large enough, $y_i$ is a small perturbation of ${\cal C}_{M}^1$ and Theorem \ref{thm-blockKron} can be applied to the polynomial $M$ and $y_i$. From this, we obtain that $y_i$ is strictly equivalent to ${\cal C}_{P}^1$ for a certain polynomial $P$ of grade $d$ and size $m\times n$, which is independent of $i$ since $y_i \in O^e(Q)$. Then $O^e(Q) = O^e ({\cal C}_{P}^1)$ and ${\cal C}_{P}^1$ has rank $r+n(d-1)$, which is equivalent to say that $P$ has rank $r$. Thus $\overline{O^e}({\cal C}_{P}^1)\supseteq \overline{O^e}({\cal C}_{M}^1)$ and $\overline{O^e}({\cal C}_{P}^1) \cap \GSYL \supseteq \overline{O^e}({\cal C}_{M}^1) \cap \GSYL$. The latter is equivalent to $\overline{O^e({\cal C}_{P}^1) \cap \GSYL} \supseteq \overline{O^e({\cal C}_{M}^1) \cap \GSYL}$ by Lemma \ref{cl}, and, by definition, is also equivalent to $\overline{O}({\cal C}_{P}^1) \supseteq \overline{O}({\cal C}_{M}^1)$, which according to the discussion after \eqref{linorb}, is equivalent to $\overline{O}(P) \supseteq \overline{O}(M)$.
The remaining part of the proof is to show that the most generic matrix pencils of size $(m+n(d-1))\times nd$ and rank $n(d-1)+r$ (see Theorem \ref{th:improvedDeDo}) are strictly equivalent to the first companion forms of $m \times n$ matrix polynomials of grade $d$ and rank $r$ if and only if these matrix pencils are strictly equivalent to the first companion form of the polynomials with the complete eigenstructures ${\cal K}_a$ in $\eqref{kcilist}$.

For each matrix polynomial $K_a$ in part (i) the $(m+n(d-1)) \times nd$ matrix pencil ${\cal C}_{K_a}^1$ has the rank $n(d-1)+r$ and by \cite{DeDM12,DeDM14} the Kronecker canonical form of ${\cal C}_{K_a}^1$ is the direct sum of the following blocks:
\begin{equation} \label{cka}
{\cal C}_{K_a}^1: \{\underbrace{L_{\alpha+d}, \dots , L_{\alpha+d}}_{s},\underbrace{L_{\alpha+d-1}, \dots , L_{\alpha+d-1}}_{n-r-s}, \underbrace{L_{\beta+1}^T, \dots , L_{\beta+1}^T}_{t}, \underbrace{L_{\beta}^T, \dots , L_{\beta}^T}_{m-r-t}\}.
\end{equation}
We show that the Kronecker canonical form of ${\cal C}_{K_a}^1$ coincides with the Kronecker canonical form of one of the most generic matrix pencils of rank $r_1=n(d-1)+r$ and size $m_1 \times n_1$, where $m_1 = m + n(d-1)$ and $n_1 = n d$, given in Theorem \ref{th:improvedDeDo}:
\begin{equation} \label{pg}
\{\underbrace{L_{\alpha_1+1}, \dots , L_{\alpha_1+1}}_{s_1},\underbrace{L_{\alpha_1}, \dots , L_{\alpha_1}}_{n_1-r_1-s_1}, \underbrace{L_{\beta_1+1}^T, \dots , L_{\beta_1+1}^T}_{t_1}, \underbrace{L_{\beta_1}^T, \dots , L_{\beta_1}^T}_{m_1-r_1-t_1}\}.
\end{equation}
Or equivalently, we show that the numbers and the sizes of the $L$ and $L^T$ blocks in \eqref{cka} and \eqref{pg} coincide, i.e., $\alpha  + d - 1 = \alpha_1$, $s=s_1$, $n-r-s = n_1 - r_1 - s_1$, $\beta=\beta_1$, $t=t_1$, and $m-r-t= m_1 - r_1 - t_1$.

For the sizes of $L$ blocks we have
\begin{align} \label{alphasize1}
\alpha  + d - 1& =\left\lfloor \frac{a}{n-r} \right\rfloor + d - 1 =  \left\lfloor  \frac{(n-r)(d-1) + a}{n-r} \right\rfloor  \\
& = \left\lfloor  \frac{(n(d-1)+r -rd) + a}{nd - (n(d-1) + r)} \right\rfloor = \left\lfloor  \frac{a_1}{n_1 -  r_1} \right\rfloor = \alpha_1,
\label{alphasize2}
\end{align}
where $a_1 = (n(d-1)+r) -rd + a$. Since $a = 0,1,\dots,rd$, then $a_1 = (n(d-1)+r) -rd,(n(d-1)+r) -rd +1,\dots,n(d-1)+r$, or equivalently $a_1 = r_1 -rd,r_1 -rd +1,\dots, r_1$.

For the numbers of $L$ blocks $s$ and $n-r-s$, we have
\begin{align*}
s &= a \mod (n-r) = ((n-r)(d-1) + a) \mod (n-r) \\
&= ((n(d-1)+r) -rd + a) \mod (nd - (n(d-1) + r))\\
&=  a_1 \mod (n_1 - r_1) =s_1
\end{align*}
and
$$n-r-s = nd - n(d-1) -r -s= n_1 - r_1 - s_1.$$

Before checking the sizes and numbers of $L^T$ blocks, note that $rd-a = rd + n(d-1)+r - n(d-1)- r -a = n(d-1)+r - (n(d-1)+r-rd +a) = r_1-a_1$ and $m-r = m + n(d-1)-(n(d-1)+r) = m_1-r_1$.
Now for $\beta$, $t$, and $m-r-t$ we have
$$\beta = \left\lfloor \frac{rd-a}{m-r} \right\rfloor = \left\lfloor \frac{r_1-a_1}{m_1-r_1} \right\rfloor  = \beta_1,$$
$$
t = (rd-a) \mod (m-r) = (r_1-a_1) \mod (m_1-r_1) = t_1,
$$
and
$$m-r-t = m + n(d-1) - n(d-1) -r -t= m_1 - r_1 - t_1.$$

Therefore ${\cal C}_{K_a}^1$ is strictly equivalent to one of the $r_1+1$ most generic matrix pencils of rank $r_1$, obtained in Theorem \ref{th:improvedDeDo}, to be exact the one with $a_1= (n(d-1)+r) -rd + a$.

The most generic pencils in Theorem \ref{th:improvedDeDo} with $a_1 < (n(d-1)+r) -rd$ are not strictly equivalent to the first companion linearization of any $m\times n$ matrix polynomial of grade $d$ since their $L$ blocks have sizes smaller than $d-1$, see \eqref{alphasize1}--\eqref{alphasize2}.

(iii) From Theorem \ref{th:improvedDeDo}-(ii) and \cite{Pokr86} we have that $\overline{O^e}({\cal C}_{K_a}^1) \not\supseteq  O^e({\cal C}_{K_{a'}}^1)$. The boundary of $O^e({\cal C}_{K_a}^1)$ is a union of (possibly infinitely many) orbits of smaller dimensions \cite[Closed Orbit Lemma, p. 53]{Bore91}, thus
$$\overline{O^e}({\cal C}_{K_a}^1)  =\left(O^e({\cal C}_{K_a}^1) \cup \bigcup_{\eta} O^e(Q_\eta) \right).$$
Therefore
$\overline{O^e}({\cal C}_{K_a}^1)  \cap O^e({\cal C}_{K_{a'}}^1) = \emptyset$, which after intersection with $\GSYL$ and applying Lemma \ref{cl} results in $\overline{O^e({\cal C}_{K_a}^1) \cap \GSYL} \ \bigcap \ \left( O^e({\cal C}_{K_{a'}}^1) \cap \GSYL \right) = \emptyset$. Thus $\overline{O^e({\cal C}_{K_a}^1) \cap \GSYL} \not\supseteq \overline{O^e({\cal C}_{K_{a'}}^1) \cap \GSYL}$, i.e., $\overline{O}({\cal C}_{K_{a}}^1) \not\supseteq \overline{O}({\cal C}_{K_{a'}}^1)$ and $\overline{O}(K_{a}) \not\supseteq \overline{O}(K_{a'})$.

(iv) By (ii) any $m\times n$ complex matrix polynomial of grade $d$ with rank at most $r$ is in one of the $rd+1$ closed sets $\overline{O}(K_a)$. Thus (iv) holds, since the union of a finite number of closed sets is also a closed set.
\end{proof}

As in the case of pencils (see Theorems \ref{th:pencgenfull1} and \ref{th:pencgenfull2}), we complete Theorem \ref{mainth} with Theorems \ref{mainth-c1} and \ref{mainth-c2}, which cover the limiting case $r=\min \{ m, n \}$ when $m\ne n$ and display the {\em unique} generic complete eigenstructure of matrix polynomials of size $m\times n$, grade $d$, and arbitrary rank. Though very simple and not surprising, we believe that Theorems \ref{mainth-c1} and \ref{mainth-c2} are stated for the first time in the literature. We only prove Theorem \ref{mainth-c2} since it implies Theorem \ref{mainth-c1} just by transposition. Observe that the ommitted case $r=\min \{m,n\}$ when $m= n$ is very simple, since, in this situation, generically a matrix polynomial of grade $d$ is regular and has all  its $nd$ eigenvalues simple.

\begin{theorem} \label{mainth-c1}
Let $m,n$, $m<n$, and $d$ be positive integers and define the complete eigenstructure ${\cal K}_{rp}$ of a matrix polynomial without elementary divisors, without  left minimal indices, and with right minimal indices $\alpha$ and $\alpha +1$, whose values and numbers are as follows:
\begin{equation}
\label{kcilist1}
{\cal K}_{rp}: \{ \underbrace{\alpha+1, \dots , \alpha+1}_{s},\underbrace{\alpha, \dots , \alpha}_{n-m-s} \} \, ,
\end{equation}
where $\alpha = \lfloor md / (n-m) \rfloor$, $s = md \mod (n-m)$. Then,
\begin{itemize}
\item[(i)] There exists an $m \times n$ complex matrix polynomial $K_{rp}$ of degree exactly $d$ and rank exactly $m$ with the complete eigenstructure ${\cal K}_{rp}$;

\item[(ii)] $\overline{O}(K_{rp}) = \POL_{d,m\times n}$.
\end{itemize}
\end{theorem}

\begin{theorem} \label{mainth-c2}
Let $m,n$, $m>n$, and $d$ be positive integers and define the complete eigenstructure ${\cal K}_{\ell p}$ of a matrix polynomial without elementary divisors, without right minimal indices, and with left minimal indices $\beta$ and $\beta +1$, whose values and numbers are as follows:
\begin{equation}
\label{kcilist2}
{\cal K}_{\ell p}: \{ \underbrace{\beta+1, \dots , \beta+1}_{t},\underbrace{\beta, \dots , \beta}_{m-n-t} \} \, ,
\end{equation}
where $\beta = \lfloor nd / (m-n) \rfloor$, $t = nd \mod (m-n)$. Then,
\begin{itemize}
\item[(i)] There exists an $m \times n$ complex matrix polynomial $K_{\ell p}$ of degree exactly $d$ and rank exactly $n$ with the complete eigenstructure ${\cal K}_{\ell p}$;

\item[(ii)] $\overline{O}(K_{\ell p}) = \POL_{d,m\times n}$.
\end{itemize}
\end{theorem}

\begin{proof} We just sketch the proof since it follows in a very simplified way the proof of Theorem \ref{mainth}. The proof of (i) follows from summing up all the indices in \eqref{kcilist2} to get
\[
t (\beta + 1) + (m-n-t) \beta = (m-n) \beta + t = nd \, .
\]
Then, \cite[Theorem 3.3]{DeDV15} guarantees that there exists an $m\times n$ matrix polynomial $K_{\ell p}$ of degree $d$, rank $n$, and with the complete eigenstructure \eqref{kcilist2}. For proving (ii), note that the first companion form $\mathcal{C}_{K_{\ell p}}^1$ has exactly the complete eigenstructure \eqref{kcilist2}, which corresponds precisely to the KCF $\mathcal{K}_{left} (\lambda)$ in \eqref{3max} if $m_1 = m + n(d-1)$ and $n_1 = nd$. Therefore, we get from Theorem \ref{th:pencgenfull2} that any $m\times n$ matrix polynomial $M$ of grade $d$ satisfies
${\cal C}_{M}^1 \in \overline{O^e}({\cal C}_{K_{\ell p}}^1)$. So, ${\cal C}_{M}^1 \in \overline{O^e}({\cal C}_{K_{\ell p}}^1) \cap \GSYL = \overline{O}({\cal C}_{K_{\ell p}}^1)$, where Lemma \ref{cl} has been used in the last equality. This proves (ii) by applying the $f^{-1}$ bijective isommetry as explained in the paragraph after \eqref{linorb}.
\end{proof}

\section{Codimensions of generic sets of matrix polynomials with fixed rank and fixed degree} \label{sec.codim}

In this section we consider the codimensions inside the space $\POL_{d, m\times n}$ of the sets of matrix polynomials $O(K_a)$, $a = 0,1,\ldots, r d,$ identified in Theorem~\ref{mainth}. More precisely, we will determine the codimensions of the orbits $\orb({\cal C}^1_{K_a})$ defined in \eqref{linorb} inside the space $\GSYL_{d,m\times n}^1$.
These codimensions provide us a necessary (but not sufficient) condition for constructing orbit closure hierarchy (stratification) graphs \cite{Dmyt15,DJKV15,DmKa14,EdEK97,JoKV13}, since the boundary of an orbit consists of orbits with higher codimensions.
Recall that, for any $P \in \POL_{d, m\times n}$, $\dim \orb^{e}({\cal C}^1_{P}) : = \dim \tsp^{e} ({\cal C}^1_{P})$ and $\cod \orb^{e}({\cal C}^1_{P}) : = \dim \nsp^{e}({\cal C}^1_{P})$, where $\tsp^{e} ({\cal C}^1_{P})$ and $\nsp^{e}({\cal C}^1_{P})$ denote, respectively, the tangent and normal spaces to the orbit $\orb^{e} ({\cal C}^1_{P})$ at the point ${\cal C}^1_{P}$, and $\dim$ and $\cod$ stand for dimension and codimension respectively. By \cite[Lemma~9.2]{JoKV13}, $\orb({\cal C}^1_{P})$ is a manifold in the matrix pencil space $\PEN_{m_1 \times n_1}$, where $m_1 = m + n (d-1)$ and $n_1 = nd$.
By \cite{DJKV15,JoKV13} we have
$
\cod \orb({\cal C}^1_{P})=
\cod \orb^{e}({\cal C}^1_{P}),
$
where the codimension of $\orb({\cal C}^1_{P})$ is considered in the space $\GSYL_{d, m\times n}^1$ and the codimension of $\orb^{e}({\cal C}^1_{P})$ in $\PEN_{m_1 \times n_1}$.
Define
$
\cod \orb (P) :=
\cod \orb({\cal C}^1_{P})\, .
$
These codimensions are computed via the Kronecker canonical form of ${\cal C}^1_{P}$ in \cite{DeEd95} and implemented in the MCS Toolbox \cite{DmJK13, Joha06}.
Note that, Theorem \ref{th-codimensions} shows that the codimensions of different $\orb (K_a)$ are distinct if $m \neq n$.

From the discussion above we have that
$
\cod \orb (K_a) =  \cod \orb^{e}({\cal C}^1_{K_a}).
$
In addition, recall that in the proof of Theorem \ref{mainth} we have seen that ${\cal C}_{K_a}^1$ is strictly equivalent to one of the $r_1+1$ most generic matrix pencils of rank $r_1$ presented in Theorem \ref{th:improvedDeDo}, to be exact the one with $a_1= (n(d-1)+r) -rd + a$. This allows us to obtain Theorem \ref{th-codimensions} from \cite[Theorem 3.3]{DeDo08}
just by replacing $m,n,r,a$ in \cite{DeDo08} by $m_1 = m + n (d-1), n_1 = nd, r_1 = r + n (d-1)$, and $a_1= (n(d-1)+r) -rd + a$, respectively, and performing some elementary simplifications that are explained in the proof below.

\begin{theorem} \label{th-codimensions}
Let $m,n,r$ and $d$ be integers such that $m,n \geq 2$, $d \geq 1$ and $1 \leq r \leq \min\{m,n\}-1$ and let $K_a, a = 0,1, \dots , rd$, be the $rd+1$ matrix polynomials with the complete eigenstructures \eqref{kcilist}.
Then the codimension of $O(K_a)$ in $\POL_{d,m\times n}$ is
$(n-r)(m(d+1)-r) + a (m-n).$
\end{theorem}
\begin{proof} Theorem 3.3 in \cite{DeDo08} yields directly that the codimension of $O(K_a)$ is $(n-r)(2m+n(d-1)-r) + ((n(d-1)+r) -rd + a) (m-n)$, which can be simplified as follows:
\begin{align*}
&(n-r)(2m+n(d-1)-r) + ((n(d-1)+r) -rd + a) (m-n)\\
&\phantom{a}=(n-r)(2m-r)+(n-r)n(d-1)+n(d-1)(m-n)-r(d-1)(m-n)+ a(m-n) \\
&\phantom{a} =(n-r)(2m-r)+(n-r)n(d-1)+(n-r)(d-1)(m-n)+ a(m-n) \\
&\phantom{a}=(n-r)(2m-r+n(d-1)+(d-1)(m-n))+ a(m-n) \\
&\phantom{a} =(n-r)((d+1)m -r)+ a(m-n).
\end{align*}
\end{proof}
\begin{remark}
In \cite[Theorem 3.3]{DeDo08} the codimension of ${\cal P}^{m\times n}_{r,1}$ is computed by taking the least codimension of all the irreducible components $\overline{\orb^e}({\cal K}_{a_1} )$ (see~\eqref{max}) of ${\cal P}^{m\times n}_{r,1}$. Since the irreducibility of $O(K_a)$ is not shown, we skip talking about the codimensions of ${\cal P}^{m\times n}_{r,d}$ in Theorem \ref{th-codimensions}.
\end{remark}
\begin{remark}[All strict equivalence orbits of the Fiedler linearizations of a matrix polynomial $P$ have the same codimensions] \label{deponlin} The results in this paper have been obtained through the first Frobenius companion linearization ${\cal C}^1_{P}$. However, it is interesting to emphasize that the same results can be obtained by using any other Fiedler linearization \cite{DeDM12}
and, in particular, note that $\cod \orb (P)$ does not depend on the choice of Fiedler linearization for any $P\in \POL_{d,m\times n}$.
The complete eigenstructures of the
Fiedler linearizations of the same matrix polynomial $P$ differ from each other only by the sizes of minimal indices \cite{DeDM12}: each left minimal index of the linearization is shifted with respect to the corresponding left minimal index of $P$ by a certain number $c(\sigma)$, which is equal for all left minimal indices. Analogously, each right minimal index of the linearization is shifted with respect to the corresponding right minimal index of $P$ by a certain number $i(\sigma)$, which is equal for all right minimal indices, and $c(\sigma) + i(\sigma)=d-1$ (we have no need to define $c(\sigma)$ and $i(\sigma)$ here but a curious reader may find the definition, e.g., in \cite{DeDM12}).
For any matrix polynomial $P$, the codimensions of all Fiedler pencils are equal to each other, see \cite[Theorem 2.2]{DeEd95} and note that $c(\sigma)$ and $i(\sigma)$ do not affect the difference of any two shifted left minimal indices ${\varepsilon_1 + c(\sigma)}$ and ${\varepsilon_2 + c(\sigma)}$, the difference of any two shifted right minimal indices ${\eta_1 + i(\sigma)}$ and ${\eta_2 + i(\sigma)}$, or any sum $\varepsilon_k + c(\sigma) + \eta_k + i(\sigma) (= \varepsilon_k + \eta_k + d-1)$.
\end{remark}

\section*{Acknowledgements}
The authors are thankful to Fernando De Ter\'an, Stefan Johansson, Bo K\r{a}gstr\"om, and Volker Mehrmann for the useful discussions on the subject of this paper.

The work of Andrii Dmytryshyn was supported by the Swedish Research Council (VR) under grant E0485301, and by eSSENCE (essenceofescience.se), a strategic collaborative e-Science programme funded by the Swedish Research Council.

The work of Froil\'an M. Dopico was supported by ``Ministerio de Econom\'{i}a, Industria y Competitividad of Spain'' and ``Fondo Europeo de Desarrollo Regional (FEDER) of EU'' through grants MTM-2015-68805-REDT and MTM-2015-65798-P (MINECO/FEDER, UE).

{\small
\bibliographystyle{abbrv}

}
\end{document}